\documentclass{article}%
\usepackage{amssymb,amsthm,amsmath,amsfonts}%,etoolbox}
\usepackage{graphicx} %
\usepackage [normalem]{ulem} %
\usepackage{subfigure}

\renewcommand{\Re}{{\rm Re}}
\renewcommand{\Im}{{\rm Im}}

\providecommand{\U}[1]{\protect\rule{.1in}{.1in}}
\newtheorem{theorem}{Theorem}[section]
\newtheorem{definition}[theorem]{Definition} 
\newtheorem{lemma}[theorem]{Lemma}
\newtheorem{proposition}[theorem]{Proposition} 
\newtheorem{corollary}[theorem]{Corollary} %
\newtheorem{remark}[theorem]{Remark} 

\newenvironment{TC} {\left \{\begin{array}{ll}} {\end{array} \right.}

\newcommand{\alp}{\alpha}

\newcommand{\lam}{\lambda}
\newcommand{\kap}{\kappa}
\newcommand{\R}{\mathbb{R}}
\newcommand{\C}{\mathbb{C}}
\newcommand{\ddt}{\, \frac{d}{dt}}

\usepackage{color}

\newcounter{margcount} %Zähler erzeugen
\title{Self-similar lifting and persistent touch-down
points in the thin-film equation}
\usepackage{authblk}
\author[1]{Carlota M. Cuesta}
\author[2]{Hans Kn\"upfer}
\author[3]{Juan J. L. Vel\'azquez}
\affil[1]{Departamento de Matem\'aticas, Faculty of Science and Technology, University of the Basque Country (UPV/EHU),\newline e-mail: carlotamaria.cuesta@ehu.eus}
\affil[2]{Institut f\"ur Angewandte Mathematik, \newline IWR, Universit\"at Heidelberg,\newline
 e-mail: hans.knuepfer@math.uni-heidelberg.de }
\affil[3]{Institut f\"ur Angewandte Mathematik, Universit\"at Bonn, \newline e-mail: velazquez@iam.uni-bonn.de}
\begin{document}

\maketitle

\begin{abstract}
  We study self-similar solutions of the thin-film equation
  \begin{align*} %
    h_{t} + (h^{m} h_{xxx})_{x} = 0 &&\text{in $\{(x,t):\quad h(x,t) > 0 \}$}
  \end{align*}
with $m \in (0,4]$, that describe the lifting of an isolated touch-down point given by an initial 
profile of the form $h_{\rm in}(x)= |x|$. This provides a mechanism for non-uniqueness of the 
thin-film equation with $m \in (2,4)$, since solutions with a persistent touch-down point 
also exist in this case. 
In order to prove existence of the self-similar solutions, we need to study a four-dimensional 
continuous dynamical system. The proof consists of a shooting argument based on the identification 
of invariant regions and on suitable energy formulas.

\medskip



\noindent
\textbf{Short title:} Self-similar lifting in the thin-film equation.

\end{abstract}

%%%%%%%%%%%%%%%%%%%%%%%%%%%%%%%%%%%%%%%%%%%%%%%%%%%%%%%%%%%%%%%%%%%%%%%%%%%%%%%%%%%%%%%%%%%%%%%%%%%%%%%%%%%%%%%%%%%
%%%%%%%%%%%%%%%%%%%%%%%%%%%%%%%%%%%%%%%%%%%%%%%%%%%%%%%%%%%%%%%%%%%%%%%%%%%%%%%%%%%%%%%%%%%%%%%%%%%%%%%%%%%%%%%%%%%
\section{Introduction}\label{sec-into}

We consider the family of thin-film equations
\begin{align} \label{tfe} %
  h_{t} + (h^{m} h_{xxx})_{x} = 0 &&\text{in $\{(x,t):\quad h(x,t) > 0 \}$}
\end{align}
for $m \in (0,4]$. In particular, we study solutions that exhibit self-similar
lifting of an initial profile with a touch-down point of the form
$h_{\rm in}(x)=a|x|$ ($a>0$) for $m\in(0,4]$. This situation is relevant for
  the merging process of two adjacent droplets. Our result also gives an
explicit example of the known non-uniqueness of weak solutions exhibited by
\eqref{tfe} as we explain below. Before we state our results and put them into
context, let us recall some properties of \eqref{tfe}.

The model \eqref{tfe} describes the spreading of a liquid film on a substrate with height profile $h(x,t)$ whose
evolution is driven by capillary effects. It can be derived formally from the Navier-Stokes equation in the so-called
regime of lubrication approximation. The exponent $m > 0$ is determined by the precise boundary conditions imposed at
the liquid--solid interface. In particular, the case $m = 3$ is related to imposing a no-slip condition at this interface, while the case $m=2$ corresponds to a Navier-slip condition \cite{Greenspan-1978,Reynolds-1886}. The
case of more general mobility $m > 0$ can be derived using certain generalized Navier-slip type conditions, see e.g.
\cite{BarrettBloweyGarcke-1998,Bertozzi-1998,GiacomelliShishkov-2005,OronDavisBankoff-1997}. Furthermore, for $m =1$,
\eqref{tfe} arises as the lubrication approximation of the Hele-Shaw flow \cite{ConstantinDupontEtal-1993}.

Equation \eqref{tfe} is a degenerate fourth order parabolic equation, formulated
on a domain with a free boundary. Note that, while the equation is parabolic on
the positivity set of $h$, the parabolicity degenerates at the points where
$h$ vanishes. Solutions of \eqref{tfe} formally satisfy a dissipation
relation, namely, for positive smooth solutions $h > 0$, it is given by
(see e.g. \cite{BernisFriedman-1990})
\begin{align}\label{energy} 
  \frac 12 \ddt \Big(\int_\R h_x^2 \ dx \Big) + \int_{\R} h^m h_{xxx}^2 \ dx \ = \ 0.
\end{align}
 
In this paper we use the following notion of weak solution of \eqref{tfe}, see also \cite{BerettaBertschDalpasso-1995,BernisFriedman-1990, BertozziPugh-1996}:
\begin{definition}[Weak Solution]\label{def-weak} 
  Let $t_0 \geq -\infty$ and let $I = [t_0,\infty)$. Then
  $h \in C^0(I \times \R)$ $\cap$ $L^\infty(I, H_{\rm loc}^1(\R))$ such that
  $h^{m/2} h_{xxx}$ exists in the distributional sense and
  \begin{align} \label{diff-fin} %
    \int_{\R} |h^{m/2} h_{xxx}|^2 \ dx < \infty &&\text{for a.e. $t > t_0.$}
  \end{align}
is called a weak solution of \eqref{tfe} if $h \geq 0$ and \text{for all $\psi \in C_c^\infty(I \times \R)$}, 
we have
  \begin{align}\label{weak}
    \int_I \int_{\R} h\psi_{t} \ dx \,ds + \int_{I} \int_{\{h>0\}} h^{m}h_{xxx}\psi _{x} dx\,ds=0.
  \end{align}
  The function $h_{\rm in} \in C^0(\R)$ is called initial data if $h(x,t) \to h_{\rm in}(x)$ as $t \to t_0$.  
\end{definition}
\begin{remark}[Dissipation] \label{rem-diss}%
  The expression $h^{m/2} h_{xxx}$ in \eqref{diff-fin} is understood in the
  distributional sense. We notice that this definition differs from the
  a.e. pointwise definition for $m \in (0,2)$. For example, for the cone
  $h(x) = |x|$ we have $h^m h_{xxx} = 0$ a.e. $x \in \R$. However, understood in
  the distributional sense, we have $h_{xxx}= \delta_{0,x}$ (the derivative of a Dirac mass at zero). By an argument based on integration by parts, it follows that $h^{\frac{m}{2}}h_{xxx} = |x|^{m/2} \delta_{0,x}$ is not well-defined if $m\leq 2$ and vanishes if $m> 2$ (see e.g. \cite{DuistermaatKolk2010}). 
This shows that $h(x)= a|x|$ with $a>0$ are stationary weak solutions of \eqref{tfe} in the sense of Definition~\ref{def-weak} if and only if $m>2$.
\end{remark}
\begin{remark}[Contact angle] %
 The Definition~\ref{def-weak} of weak solutions does not include a contact
  angle condition. Including such a condition for weak solutions is not
  straightforward since their regularity is not sufficient to control the 
first derivative point-wise. We note that there are other definitions of
  weak solutions. In particular, there is one that requires additional regularity 
(which leads to the so-called entropy solutions) that implies zero contact angle 
(see, for instance, \cite{BerettaBertschDalpasso-1995,BernisFriedman-1990,BertozziPugh-1996} -one
 space dimension- and \cite{BertschDalpassoGarckeGruen-1998, DalpassoGarckeGruen-1998} 
-several space dimensions). There are also formulations of weak solutions for a prescribed
 non-zero contact angle. This is a more delicate issue, since the corresponding solutions are less
  regular. We refer the reader to \cite{Otto-1998} for $m = 1$, and to 
\cite{BertschGiacomelliKarali-2005} for a different notion of weak solution that modifies the 
energy \eqref{energy} and is valid for $m \in (0,3)$.
\end{remark}

In the context of \eqref{tfe}, the {\em rupture} of a droplet into two droplets
and, reversely, the {\em merging} of two droplets into a single droplet, 
corresponds to a topological change of the positivity set $\{ h > 0 \}$. In this
work, we consider the situation when the positivity set initially consists of
two connected components which are separated at a single {\em touch-down point}
in the specific case $h_{\rm in}(x) = a |x|$. For this initial data and for
$m \in (0,4)$, we show that there exist self-similar solutions which display
{\em lifting} of the touch-down point. These solutions are weak solutions in the
sense of Definition~\ref{def-weak}. For completeness, we also consider the case
$m=4$ which corresponds to lifting from an 'initial' negative infinite time. On
the other hand, $h(x)=a|x|$ ($a>0$) is a stationary weak solution of \eqref{tfe}
for $m > 2$ (see Remark \ref{rem-diss}); hence our results yield a special
mechanism of non-uniqueness for solutions of \eqref{tfe}.

The non-uniqueness of weak solutions for $m \in (0,5)$ has been shown
by Beretta, Bertsch and Dal Passo in \cite{BerettaBertschDalpasso-1995}. In
order to show this, the authors regularize \eqref{tfe} and this
allows to construct solutions which stay positive, while other solutions 
for the same initial data vanish at some specific points in space. The self-similar solutions
obtained in the current paper yield a mechanism on how this non-uniqueness may
take place at isolated touch-down points (we refer also to the discussion at
  the end of Section \ref{sec-results}). For existence theory of weak
  solutions of \eqref{tfe}, we refer to, e.g., 
\cite{BerettaBertschDalpasso-1995, BernisFriedman-1990,BertozziPugh-1996} 
in the one-dimensional case, and to 
\cite{BertschDalpassoGarckeGruen-1998,DalpassoGarckeGruen-1998,Gruen-2004-1}
 in the case of higher space dimensions. Corresponding results on
 existence and uniqueness of classical solutions have been addressed in, e.g.,
 \cite{GiacomelliGnannKnuepferOtto-2014,GiacomelliKnuepfer-2010,GiacomelliKnuepferOtto-2008,John-2015,Knuepfer-2011}, see also \cite{Gnann-2016}. A well-posedness result for $m=1$ in the case of partial wetting is included in \cite{KnuepferMasmoudi-2013,KnuepferMasmoudi-2015}.

Most of the analysis on self-similar solutions of \eqref{tfe} 
has focused on the study of source-type solutions. In, e.g., 
\cite{BernisPeletierWilliams-1992}, it has been shown that source-type self-similar 
solutions exist for all $m \in (0,3)$, but not for $m = 3$; for a corresponding result 
in higher space dimensions, we refer to \cite{BernisFerreira-1997}. These results are 
consistent with the conjecture that droplet spreading is not possible for $m \geq 3$ (see
\cite{DussanDavis-1974,HuhScriven-1971}). Additional regularity up to the moving 
boundary of the source-type self-similar solutions with $m\in(\frac 32,3)$ is
studied in \cite{GiacomelliGnannOtto-2013}. A source-type self-similar solution
with zero-contact angle and where a portion of the boundary may undergo drainage is
studied in \cite{BernisHulshofKing-2000} for $m \in (0,3)$.

\medskip

\textbf{Structure of the paper.} In Section~\ref{sec-results}, we state our main
results on the existence of solutions which exhibit self-similar lifting. The
proof of the theorem is given in Section~\ref{sec-lifting}. In the Appendix, we
formally derive the behavior of the profile for two identical droplets merging
in a self-similar way.

\section{Statement of the results}\label{sec-results}
We consider initial data with an isolated touch-down point at the origin, i.e.
\begin{align}\label{separation-point}
  h_{\rm in}(0) = 0 \qquad \text{and} \qquad
  h_{\rm in}(x) > 0 \quad \mbox{for}\ x \neq 0.
\end{align}
Moreover, we assume that the initial data is symmetric around this point, $h_{\rm in}(-x) = h_{\rm in}(x)$, with
non-zero slope on both sides, i.e.
\begin{align}\label{separation-point-angle}
  \lim_{x\to 0} |h_{\rm in,x}(x)| = a \quad && \mbox{for some} \quad a>0.
\end{align}
\begin{figure}[hhh]
  \begin{center}
    \includegraphics[width=8cm,height=7cm]{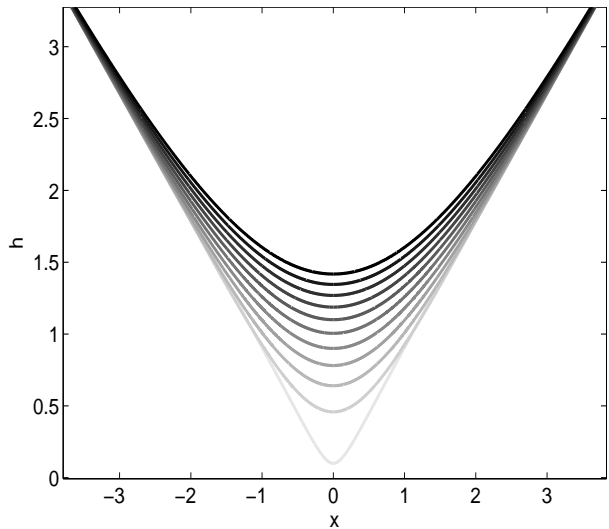}
  \end{center}
  \caption{The picture shows solutions $h(x,t)=t^\alpha f(|x|/t^\alp)$, where
    $f$ satisfies \eqref{ode}--\eqref{LBU} with $m=2$ and for $t\in[0.01,1.01)$
    at increments of size $0.2$. The solution is calculated using the ODE
      solver ode45 of MATLAB (a fourth order Runge-Kutta algorithm) by iterating
      in the shooting parameter $\kappa$.}
\label{evolution}
\end{figure}
We consider the existence of a self-similar solution of the form
\begin{align}\label{ss-form}
  h(x,t) = t^{\frac{1}{4-m}} f(y)\,,\quad y=\frac{|x|}{t^{\frac{1}{4-m}}} && \mbox{if } m \in (0,4)
\end{align}
for $x \in \R$ and for some $f : [0,\infty) \to \R$. In order to preserve
  the linear behavior of solutions in the outer region away from the touch-down
  point, we require (see below) that $f$ grows linearly as $|y| \to \infty$. For
$m=4$, we also consider solutions of the form:
\begin{align}\label{ss-form4}
  h(x,t) = e^{tb} f(y)\,,\quad y=\frac{|x|}{e^{tb}},
  \end{align}
  where $b > 0$ is a free parameter. For $m \in (0,4)$, the solution is defined
  for $t > 0$, while the problem is formulated for $t \in\R$ when 
$m = 4$. With this ansatz and defining the parameter
\begin{equation}\label{alpha:def}
\alpha:=
\begin{cases}
\displaystyle{\frac{1}{4-m}} \qquad\qquad  &\mbox{if } m \in (0,4),  \vspace{0.4ex}\\
\ \ \ b &\mbox{if } m=4, 
\end{cases} 
\end{equation}
the equation \eqref{tfe} becomes the fourth order ODE
\begin{align}\label{ode} 
  \alp \left(f - y f_y \right)+ (f^m f_{yyy})_y = 0 && \text{for $y \in (0,\infty)$.}
\end{align}
By our assumption that the solution is even, we have $f_{y}(0)=f_{yyy}(0)=0$, which leaves 
the two free parameters $f(0)$ and $f_{yy}(0)$. In view of the scaling invariance 
$(y, f) \mapsto (\lam y, \lam^{{\frac 4m}} f)$ for any $\lam > 0$ of
\eqref{ode}, it is enough to consider the initial conditions $f(0)=1$, $f_{yy}(0)=\kappa$ for some 
$\kappa \in \R$. Summarizing, we look for a solution of \eqref{ode} which satisfies the initial conditions
\begin{align}
  f(0)=1,\quad  f_y(0)=0, \quad f_{yy}(0)=\kappa>0, \quad f_{yyy}(0) = 0, \label{init:ss}
\end{align}
and which satisfies the behavior
\begin{align} 
 &\frac{f(y)}{y} \rightarrow a > 0  && \text{as} \quad y\to\infty, \label{LBU}
\end{align}
for some $a > 0$ which is not prescribed {\it a priori}.

\medskip

Our main result states the existence of a solution for the problem \eqref{ode}--\eqref{LBU} for $m \in(0,4]$:
\begin{theorem}[Existence of a self-similar solution] \label{thm-lifting}
  \text{} %
  \begin{enumerate}
  \item  Let $m \in (0,4)$. Then there exists $\kappa\in (0,\sqrt{12\alpha})$ and $a>0$ such that 
the problem \eqref{ode}--\eqref{LBU} has a solution.
  \item Let $m = 4$ and let $b > 0$. Then there exists $\kappa\in (0,\sqrt{12\alpha})$ and $a>0$ 
such that the problem \eqref{ode}--\eqref{LBU} has a solution that satisfies
    \begin{align}
      f(y) = ay +  \Re [K y^{1+z_0}] + o(y^{1+\Re( z_0)}) &&\text{as $y \to \infty$}
    \end{align}
    for $z_0, K \in \C$ with $\Re( z_0) < -1$. Here, $z_0$ is the root of
    $(1-z^2)(2+z) + \frac{b}{a^4}$ with the largest negative real part and, for definiteness, with $\Im(z_0)\geq 0$.%, in case of two complex conjugates.
  \end{enumerate}
\end{theorem}
Using the scaling invariance of the equation, for $m \in (0,4)$ we obtain a
solution $f$ of \eqref{ode}--\eqref{LBU} for any prescribed slope $a$ at
infinity (the corresponding solution does not necessarily satisfy the condition
$f(0) = 1$). For $m = 4$, the same type of rescaling leaves the slope at
infinity invariant. However, by rescaling in $y$ but not in $f$ (preserving the
condition $f(0) = 1$) we obtain a solution of the rescaled equation for some $b > 0$,
 with the desired slope at infinity. As a consequence of
Theorem~\ref{thm-lifting}, the thin-film equation \eqref{tfe} with initial data
$h_{\rm in} (x)= a|x|$ admits a solution which exhibits self-similar lifting for
any prescribed slope at infinity:
\begin{corollary}[Self-similar lifting for the thin-film equation] \label{cor-lifting} \text{}%
  \begin{enumerate}
  \item Let $m \in (0,4)$. Then, for any $a>0$, there exists a weak solution $h$ of \eqref{tfe} 
with initial data $h_{\rm in}(x)=a|x|$ of the form \eqref{ss-form}.
  \item Let $m = 4$. Then, for any $a>0$, there exists $b > 0$ and a self-similar
    weak solution $h$ of \eqref{tfe} of the form \eqref{ss-form}, defined for
    all $t \in \R$, and such that
  \begin{align*}
    h(x,t) =  a|x| + \Re \Big[K \Big(\frac{|x|}{e^{bt}} \Big)^{1+z_0}\Big]  + o\Big(\Big(\frac{|x|}{e^{bt}} \Big)^{1 + \Re( z_0)} \Big)
  \end{align*}
  uniformly as $t \to - \infty$ for $x \in [-R,R]$ for any fixed $R >0$. 
Here, $K, z_0 \in \C$ are the same constants as in Theorem~\ref{thm-lifting}.
  \end{enumerate}
\end{corollary}

It is an open question whether the self-similar solution in Theorem~\ref{thm-lifting} is
unique. 

For $m \in (0,4)$, the solution in Corollary~\ref{cor-lifting} describes lifting
at time $t = 0$.  For $m=4$, the solution is defined for all $t \in \R$ and
  describes lifting in infinite time, in the sense that the cone $h(x)=a|x|$
is approached in the limit $t \to -\infty$. Note that the
corresponding solutions do not represent lifting if $m>4$, but instead, such
solutions would convergence to the cone as $t\to\infty$. In
fact, we do not expect that lifting is possible in this case. 
It is an open question whether the self-similar solution in Theorem~\ref{thm-lifting} is
unique.

Figure~\ref{evolution} illustrates the evolution given by the
self-similar solutions for the case $m=2$. 

Finally, we recall that the self-similar solutions obtained in this
paper can be understood as locally describing the merging of two adjacent 
droplets. In the Appendix, we present the formal asymptotics of two
initially stationary identical droplet that meet at an isolated
touch-down point. The construction of these solutions should follow from a
combination of the arguments in this paper with a suitable localization
argument.

\medskip

\textbf{Strategy of the proof.} The proof of Theorem~\ref{thm-lifting} is based on a shooting argument
 with shooting parameter $\kappa=f_{yy}(0)$. We first identify two invariant regions $\Sigma_\pm$. 
We show that any solution which enters the invariant $\Sigma_-$ region (related to small $\kap$) exhibits 
touch-down to zero at some finite $y^{\ast}>0$, i.e.
\begin{align}\label{TDB}
  f(y^{\ast}) = 0 && \text{and} && 
  f(y) >0 \quad\mbox{for all } y \in (0,y^{\ast})
\end{align}
for some $y^{\ast} \in(0,\infty)$. On the other hand, solutions which enter the
invariant region $\Sigma_+$ (related to large $\kappa$) exhibit at least
quadratic growth, i.e.
\begin{align}\label{FQBU} %
  f(y) \geq c y^2 && 
  && \text{for all $y \in (0,\infty)$}
\end{align}
and for some $c > 0$. We then invoke the energy formula
\begin{align}\label{bas-en}
  \frac{d}{dy}\Big( \alp\Big(f f_{y}-\frac{1}{2}y (f_{y})^2\Big) + f^{m}f_{yy}f_{yyy} \Big)
  =\frac{\alp}{2} f_{y}^2  + f^{m} f_{yyy}^{2},
\end{align}
(obtained by multiplying \eqref{ode} with $f_{yy}$) and another energy
  identity (see Lemma~\ref{lem-E1}) to show that there are solutions that fall 
into neither of these invariant regions. The proof is then concluded by
showing that these solutions necessarily satisfy the linear growth \eqref{LBU}.

\medskip

\textbf{Discussion:} The self-similar solutions found in the current work give a
specific mechanism of non-uniqueness related to topological changes of the
positivity set of the solution. Namely, we show that if $m\in (2,4)$ there are
at least two weak solutions, one that is the stationary solution
$h_{\rm in}(x)=a|x|$ and another solution exhibiting self-similar lifting. We note that, since the initial
condition for the self-similar solution represent stationary solutions, the
  lifting can in principle occur at any time, thus giving an infinite set of
solutions with the same initial data.  We believe that for other initial data
with, e.g., $h_{\rm in}(x) \approx |x|^{\beta}$ for small $|x|$, other type of 
self-similar lifting solutions may also exist.

  \medskip

  For weak solutions in the sense of Definition~\ref{def-weak} with initial data 
$h_{\rm in}(x)=a|x|$, our result suggests the following picture:
  \begin{enumerate}
  \item If $m \in (0,2]$, there exists a unique weak (self-similar) solution that lifts up.
  \item If $m \in (2,4)$, there exists a stationary solution as well as infinitely many 
solutions which exhibit lifting at arbitrary time.
  \item If $m \geq  4$, there exists a unique weak solution that is $h_{\rm in}(x)=a|x|$.
  \end{enumerate}
  The statements (i) and (iii) are conjectures while (ii) follows from Corollary \ref{cor-lifting}.  
While the stationary solutions $h(x)=a|x|$ for $m \in (2,4)$ are weak solutions with zero dissipation, 
our self-similar solutions have finite positive dissipation (see Corollary~\ref{corol-diss}). 
This raises the natural question as to whether the criterion of maximum dissipation yields uniqueness 
of weak solutions.

It is now pertinent to compare our results with the non-uniqueness result of weak solutions 
by Beretta, Bertsch and Dal Passo. More precisely, in \cite{BerettaBertschDalpasso-1995}, 
the authors show non-uniqueness for weak solutions of \eqref{tfe} on a finite 
interval for $m \in (0,5)$. They show the existence of entropy solutions for 
$m \in (0,5)$ which converge to their mean as $t\to\infty$ and become strictly positive in finite time 
for non-trivial initial data. Furthermore, they show that for any $m>0$ there exist weak solutions satisfying 
\eqref{energy}--\eqref{weak}, there are ones that preserve touch-down points in space and time if 
$m \in (0,3)$, and others that have fixed compact support in space and time if $m\geq3$.
 As a consequence, weak solutions with finite dissipation, i.e. solutions satisfying 
\eqref{weak} and \eqref{energy}, are not unique if $m \in (0,5)$. In particular, Theorem~3.2 in
 \cite{BerettaBertschDalpasso-1995} implies that for $m\geq 3$ there exists a
 positive solution if the initial data $h_{\rm in}$ is sufficiently large near
 the touch--down points. For initial data of the form $|x|^{\beta}$, the
 condition is satisfied if $\beta \in (0,\frac 1{m-3})$ for $m \neq 3$ (and
 arbitrary $\beta$ if $m = 3$). It would be interesting to show whether lifting
 can be realized by self-similar lifting with such initial data.  On the other
 hand, uniqueness of entropy solutions is still an open question (see also
 \cite[Theorem~5.1 and 5.2]{BertschGiacomelliKarali-2005}).

%%%%%%%%%%%%%%%%%%%%%%%%%%%%%%%%%%%%%%%%%%%%%%%%%%%%%%%%%%%%%%%%%%%%%%%%%%%%%%%%%%%%%%%%%%%%%%%%%%%%%%%%%%%%%%%%%%%
\section{Proof of the results} \label{sec-lifting}

%Throughout this section we assume $m\in(0,4]$. 

\subsection{Reformulation of the problem and invariant regions}

We write \eqref{ode} as a system of four first order equations, where the new
variables are adapted to the scaling invariance of \eqref{ode}.  Namely, we
introduce
\begin{align}\label{def-PWZQ}
  \Phi:=\frac{f}{y^{\frac{4}{m}}}, \quad W:=\frac{yf_{y}}{f},\quad Q:=y^{\frac{2}{3}}f^{\frac{m-3}{3}}f_{yy},\quad
  Z:=y^{\frac{1}{3}}f^{\frac{2m-3}{3}}f_{yyy},
\end{align}
and the independent variable $\xi\in\mathbb{R}$ by
\begin{align}\label{def-xi}
  y=e^{\xi}\,.
\end{align}
In terms of $(\Phi,W,Q,Z)$ and $\xi$, \eqref{ode} turns into the following system 
of four first order ordinary differential equations:
\begin{align}
  \begin{cases}
    \displaystyle \frac{d\Phi}{d\xi} = \Phi\Big(W - \frac{ 4}{m}\Big), \vspace{1.2ex} \\
    \displaystyle
    \frac{dW}{d\xi } = \frac{Q}{\Phi^{\frac{m}{3}} }+ W ( 1 - W ), \vspace{1.2ex} \\
    \displaystyle \frac{dQ}{d\xi } = \Big( \frac{2}{3}+\frac{m-3}{3}W \Big) Q + \frac{Z}{\Phi^{\frac{m}{3}}}, \vspace{1.2ex} \\
    \displaystyle \frac{dZ}{d\xi } = \alp \frac{W-1}{\Phi^{\frac{m}{3}}} + \Big(\frac{1}{3} - \frac{m+3}{3} W\Big)Z.
  \end{cases}
  \label{DS}%
\end{align}
We use the compact notation $X=(\Phi,W,Q,Z)$ to denote any point in the
phase-space associated to the dynamical system \eqref{DS}. We shall also write
$X(\cdot;\kappa)$ as the solution of \eqref{DS} that corresponds to the solution
of \eqref{ode}--\eqref{init:ss} when we want to explicitly indicate the
dependence on the value of $\kappa$.

\medskip

We find the following invariant regions of \eqref{DS}:
\begin{lemma}[The invariant regions $\Sigma_{\pm}$] \label{lem-invariant} \text{}%
  \begin{enumerate}
  \item The domains
    \begin{align*}
      \Sigma_{\pm} := \left\{ (\Phi, W,Q,Z) \in\R^{4} : \Phi > 0, \ \pm (W-1)> 0,\ \pm Q>0,\ \pm Z>0 \right\}
    \end{align*}
    are invariant regions for the dynamical system \eqref{DS}.
  \item If $X(\xi) \in\Sigma_{+}$ for some $\xi \in\R $, then the solution 
exhibits the behavior \eqref{FQBU}.
  \item If $X(\xi) \in\Sigma_{-}$ for some $\xi \in\R $, then the solution
    exhibits the behavior \eqref{TDB}.
  \end{enumerate}
\end{lemma}
\begin{proof} %
  {\it (i)}: Suppose that $X(\xi_0) \in \Sigma_+$ for some $\xi_0 \in \R$. 
We first note that by the first equation in \eqref{DS}, the condition 
$\Phi > 0$ is preserved for $\xi > \xi_0$. From the second equation in
 \eqref{DS} and since $Q > 0$, we get $W_\xi \geq W(1-W)$ which implies that
 the property $W > 1$ is preserved as long as $Q>0$. From the third equation in 
\eqref{DS} and since $Z > 0$, we get that the property $Q > 0$ is preserved as
 long as $Z > 0$. From the fourth equation in \eqref{DS} and
 since $W>1$, we get $Z_\xi \geq (\frac 13 -(1+ \frac m3)W)Z$ which implies
 that the property $Z> 0$ is preserved as long as $W>1$. This shows that the
 region $\Sigma_+$ is invariant, i.e. $X(\xi) \in \Sigma_+$ for all
 $\xi \geq \xi_0$. The argument for the invariance of $\Sigma_-$ proceeds analogously.

  \medskip

  {\it (ii)}: Suppose that $X \in \Sigma_+$ for all $\xi > \xi_0$ and some 
$\xi_0 \in \R$. Since $Q>0$ and $Z>0$, we have $f_{yyy}>0$ and $f_{yy}>0$
 by the definition of $Z$ and $Q$, and hence $f_{yy}(y)>f_{yy}(y_0) > 0$ for
 all $y\geq y_0 = e^{\xi_0}$. In particular, $f \geq C y^2$ for all
 $y \geq y_0$ for some $C > 0$ which depends on $y_0$. Since $f > 0$ and
 $f$ is continuous in $(0,y_0)$, we obtain \eqref{FQBU}.

  \medskip
  
 {\it (iii)}: Suppose that $X\in\Sigma_-$ for all $\xi \geq \xi_0$ and some
 $\xi_0 \in \R$. Using the first equation in \eqref{DS}, we calculate
\begin{align} \label{ii-phi} %
    \Phi^{\frac m3}(\xi) %
    \stackrel{\eqref{DS}}= \Big(\Phi(\xi_0) \exp \Big(\int_{\xi_0}^\xi \left(W(s)-\frac 4m\right)ds \Big) \Big)^{\frac m3} %
  \\  = C \exp \Big( \int_{\xi_0}^\xi \left(\frac m3 W(s)-\frac 43\right)ds \Big).
\end{align}
  Since $Z<0$, from the third equation in \eqref{DS} we obtain
 \begin{align} \label{ii-Q} %
    Q(\xi)<-C \exp \Big( \int_{\xi_0}^\xi \left(\frac 23 + \frac {m-3}3 W(s) \right)ds \Big) < 0 &&\text{for $\xi\geq \xi_0$}.
 \end{align}
  Using the second equation in \eqref{DS} together with \eqref{ii-phi} and \eqref{ii-Q}, we obtain
  \begin{align*}
    \frac{dW(\xi)}{d\xi} < -C \exp \Big(\int_{\xi_0}^\xi (2-W(s))ds \Big) +W(\xi)(1-W(\xi)) && \text{ for $\xi\geq\xi_0$}.
  \end{align*}
  Since $W<1$, this implies $W_\xi \leq -Ce^{\xi}$ for $\xi$ sufficiently
  large. It follows that $W < 0$ eventually and hence $f_y <0$ for sufficiently
  large $y = e^\xi$. Since $Q < 0$ and hence $f_{yy} < 0$, it follows that
  there is $y^* > 0$ with $f(y^*) = 0$.
\end{proof}
The next lemma relates the value of $\kappa$ to the invariant regions
$\Sigma_\pm$:
\begin{lemma}[The case of large and small $\kappa$] \label{lem-kapextreme} \text{} %
  \begin{enumerate}
  \item If $\kappa<0$, then $X(\cdot;\kap)$ enters $\Sigma_{-}$, i.e. there is $\xi_0 \in \R$ with $X(\xi_0;\kap) \in\Sigma_{-}$.
  \item If $\kappa >\sqrt{12\alpha}$, then $X(\cdot;\kappa)$ enters $\Sigma_{+}$, i.e. there is $\xi_0 \in \R$ with
    $X(\xi_0;\kappa) \in\Sigma_{+}$.
  \end{enumerate}
\end{lemma}
\begin{proof} 
  {\it (i):} If $\kappa < 0$, \eqref{ode} and \eqref{init:ss} imply that for a small enough 
$y_0 = e^{\xi_0} > 0$ we have $W(\xi_0) < 0$, $Q(\xi_0) < 0$ and $Z(\xi_0) <0$ and hence 
$X(\xi_0;\kap) \in \Sigma_-$.

  \medskip

{\it (ii):} By integrating \eqref{ode} once, we get $f^m(y)f_{yyy}(y) =\alp \int_0^y( sf_s(s) - f(s)) ds$. 
Since  $(y f_{y} - f)_y = yf_{yy}$ and $(yf_y - f)(0) = -1$, we obtain
  \begin{align}\label{eq-int}
    f_{yyy}(y) =\frac{\alp}{f^m(y)} \Big(\int_0^y \int_0^s r f_{rr}(r) \, dr ds  - y \Big)\,.
  \end{align}
Let $\overline y \in (0,\infty]$ be the largest value such that the
 solution is defined and $f_{yy}(y) > \frac \kap 2$ for all
 $y \in I := (0,\overline y)$. For $y \in I$, we have $f_y(y) > 0$ and 
$f(y) > 1$. By \eqref{eq-int} we also have $f_{yyy}(y) > -\alpha y$ and hence 
$f_{yy}(y) \geq \kap -\alpha \frac{y^2}2$ for $y \in I$. In turn, this implies 
$y^* := \sqrt{\kap /\alpha} \in I$. We thus have 
$f_{yy}(y) \geq \frac \kap 2$ for all $y < y^*$ and hence
\begin{align} %
    f_{yyy}(y^*) %
    &
\stackrel{\eqref{eq-int}}{\geq}
\frac{\alp y^* }{f^m(y^*)} \Big( \frac{\kap^{2}}{12\alpha} - 1 \Big) > 0 \quad
    \text{if}\quad \kap >\sqrt{12\alpha}.
\end{align}
This implies $Z(\xi^*) > 0$ where $y^* = e^{\xi^*}$. Clearly, we have 
$f_{yy}(y^*) > 0$ and hence $Q(\xi^*) > 0$. We also calculate
\begin{align*} 
    y^* f_y(y^*) - f(y^*) =
    \int_0^{y^*} s f_{ss} \ ds - 1 \geq \frac {\kap^2 }{4\alpha}
    - 1 > 0 &&\text{if $\kap > \sqrt{12\alpha}$, }
\end{align*}
which implies $W(\xi^*) > 1$ by the definition \eqref{def-PWZQ} of $W$. 
The above calculations hence yield $X(\xi^*;\kappa) \in \Sigma_+$ if $\kappa >\sqrt{12\alpha}$.
\end{proof}
We can now conclude that there exist values of $\kappa$ such that the
corresponding solutions to \eqref{ode}--\eqref{init:ss} enter neither of the
invariant regions:
\begin{lemma}[Shooting argument] \label{lem-shooting} %
  Let
  \begin{align*}
    & I_{\pm} := \left\{ \kappa\in\mathbb{R}:X\left( \xi;\kappa\right) \in
      \Sigma_{\pm}\ \text{for some }\xi>0\right\}, %\\
    && I_{0} := \R \backslash(I_{+} \cup I_{-}).
  \end{align*}
Then, the sets $I_{\pm} \subset\R$ are open, nonempty and disjoint and, hence, $I_{0} \neq\emptyset$.
\end{lemma}
\begin{proof}
By definition, the sets $\Sigma_\pm$ are open. By the continuous dependence of 
solutions on the initial data, the sets $I_\pm$ are hence also open. By 
Lemma~\ref{lem-invariant}, the sets $I_\pm$ are disjoint. By Lemma~\ref{lem-kapextreme}, 
the sets $I_\pm$ are non-empty. 
%Finally, the statement $I_0 \neq \emptyset$ follows since $\mathbb{R}$ is connected.
\end{proof}

In order to conclude the proof of Theorem~\ref{thm-lifting} we have to characterize the behavior 
as $\xi\to\infty$ of the solutions with initial condition in $I_0$ and show that this corresponds to 
\eqref{LBU}. We do this in the next two sections.

\subsection{Energy formulas and corollaries} \label{sec-energies} %

The following two energy formulas are essential for the proof of Theorem~\ref{thm-lifting}:
\begin{lemma} [Energy Formula I]\label{lem-E1} %
We define $E_{1}(y)$ by
\begin{align}\label{E1-y}
    E_{1}(y) \ := \alp\Big(f f_{y}-\frac{1}{2}y f_{y}^2\Big) + f^{m}f_{yy}f_{yyy}.
\end{align}
Then, if $f$ is a smooth solution of \eqref{ode}, $E_{1}$ is an increasing quantity with
\begin{align}\label{D1-y}
    \frac{d}{dy} E_{1}(y) =\frac{\alp}{2}f_{y}^{2} + f^{m} f_{yyy}^{2}\,.
\end{align}
In particular, if $f$ solves \eqref{ode}--\eqref{init:ss}, then 
$E_1(0) = 0$ and $E_1(y) > 0$ for all $y > 0$.
\end{lemma}
\begin{proof}
Multiplying \eqref{ode} by $f_{yy}$ we get
\begin{align*}
    0 &= \alp\left(f f_{yy} - yf_y f_{yy} \right)+ (f^m f_{yyy})_y f_{yy} \\ %
    &= \big( \alp\big(f f_{y} - \frac 12 y f_y^2 \big)+ f^m f_{yyy} f_{yy}\big)_y + \alp\big(- f_y^2 +
      \frac 12 f_y^2\big) - f^m f_{yyy}^2,
\end{align*}
and rearranging terms we obtain the result.
\end{proof}

\begin{lemma}[Energy Formula II]\label{lem-E2}
We define $E_{2}(y)$ by
\begin{align}\label{E2y}
  E_{2}(y)\ := - \frac{\alp}{2y} \left(f - y f_y \right)^2 + f^{m}f_{yy}f_{yyy}.
\end{align}
Then, if $f$ is a smooth solution of \eqref{ode}, $E_{2}$ is an increasing quantity with
\begin{align}\label{D2y} %
  \frac{d}{dy}E_{2}(y) = \frac{\alp}{2 y^2}\left( f - y f_{y} \right)^2 + f^{m}f_{yyy}^{2}.
\end{align}
In particular, if $f$ is a solution of \eqref{ode}--\eqref{init:ss} then $\lim_{y\to 0^+} E_2(y)=-\infty$.
\end{lemma}
\begin{proof}
We start from the energy formula of Lemma~\ref{lem-E1}. We calculate
  \begin{align} \label{sq-term} %
    \alp \frac d{dy} \Big(- \frac{f^2}{2y}\Big) = \alp \Big( - \frac{f f_y}{y} + \frac{f^2}{2y^2} \Big).
  \end{align}
We then take the sum of this identity with the first energy equality \eqref{D1-y},  %
\begin{equation*} %
    \frac{d}{dy} \Big( \alp \big(f f_{y}-\frac{1}{2}y f_{y}^2\big) + f^{m}f_{yy}f_{yyy} \Big) %
    =\frac{\alp}{2}f_{y}^{2} + f^{m}f_{yyy}^{2}.
\end{equation*}
Taking the sum completes the squares in both sides of the resulting equation 
and thus yields the assertion of the lemma.
\end{proof}
\begin{remark}\label{E:DS}
In the variables \eqref{def-PWZQ} and \eqref{def-xi}, the energy formula for $E_{1}$ of Lemma~\ref{lem-E1} becomes
  \begin{align}
    \frac{d}{d\xi}\Big(e^{\frac{8-m}{m}\xi}\Phi^{2}\big(-\frac{\alp}{2}W(W-2)+QZ\big) \Big)
    = e^{\frac{8-m}{m}\xi}\Phi^{2} \Big( \frac{\alp}{2} W^2+ \frac{Z^{2}}{ \Phi^{\frac3m}}\Big)
    \label{E1}
  \end{align}
  and the energy formula for $E_{2}$ of Lemma~\ref{lem-E2} takes the form
  \begin{align}
    \frac{d}{d\xi} \Big( e^{\frac{8-m}{m}\xi}\Phi^{2}\big(-\frac{\alp}{2} ( W - 1 )^{2} + QZ\big) \Big) =
    e^{\frac{8-m}{m}\xi} \Phi^{2} \Big( \frac{\alp}{2}(W - 1)^{2}+ \frac{Z^{2}}{ \Phi^{\frac3m}} \Big).
    \label{E2}%
  \end{align}
\end{remark}
With the aid of these energies we now show the following lemma:
\begin{lemma} \label{lem-bounds} %
  Let $\kap \geq 0$ and
  $X(\cdot;\kap)=(\Phi,W,Q,Z)$ be the corresponding solution of \eqref{DS}.
  Then:
  \begin{enumerate}
  \item If $W(\xi^{*}) <0$ for some $\xi^{*} \in \R$, then $X(\xi^{*};\kap)\in\Sigma_{-}$.
  \item If $W(\xi_{\ast})>2$ for some $\xi^{*} \in \R$, then $X(\xi^{*};\kap) \in\Sigma_{+}$.
  \end{enumerate}
\end{lemma}
\begin{proof}
  If $\kap = 0$, then the solution satisfies $W(\xi)=Q(\xi)=Z(\xi)=0$ for all $\xi\in\mathbb{R}$
 and there is nothing to prove. Hence, in the following we assume that $\kappa >0$. 
By the initial conditions, we have $W(\xi) \to 0$ as $\xi \to -\infty$ and since $\kappa > 0$, 
we also have $W(\xi) \in (0,2)$ for $\xi$ sufficiently small.

\medskip

{\it (i):} Suppose that $W(\xi_{\ast }) <0$ for some $\xi_\ast \in \R$ and let
\begin{align}
      \bar{\xi} = \min \left\{ \xi \in\mathbb{R}: \ W(\xi) = 0 \right\}.
\end{align}
Since $W \in (0,2)$ for $\xi$ sufficiently small and since $W$ is continuous, 
the minimum exists and we have $-\infty < \bar{\xi}<\xi_*<\infty$. By the definition of $\bar{\xi}$, 
we have $W_\xi(\bar{\xi}) \leq 0$. On the other hand, by the second equation in \eqref{DS}, 
we have $W_\xi(\xi)=Q(\bar{\xi})/\Phi^{\frac{m}{3}}(\bar{\xi})$ and hence $Q(\bar{\xi})\leq 0$. 
By Lemma~\ref{lem-E1}, we have $E_1(y) > 0$ for all $y > 0$. In view of Remark \ref{E:DS}, this implies
\begin{align}\label{main:ineq2}
 -\frac{\alp}{2} W(\bar \xi)( W(\bar \xi)-2) + (QZ)(\bar \xi)=  (QZ)(\bar \xi)> 0 .
\end{align}
Since $Q(\bar \xi) \leq 0$, this yields $Q(\bar \xi) < 0$ and $Z(\bar{\xi})< 0$. 
This finally implies that $X(\bar{\xi};\kap)\in \Sigma_-$ and Lemma~\ref{lem-invariant} 
implies the assertion.

  \medskip

{\it (ii):} Suppose now that $W(\xi_*) >2$ for some $\xi_* \in \R$. We define
  \begin{align}
    \bar{\xi}=\min \left\{ \xi\in \mathbb{R}: \ W(\xi) = 2 \right\}.
  \end{align}
  Since $W \in (0,2)$ for $\xi$ sufficiently small and since $W$ is continuous, the minimum
 exists and we have $-\infty < \bar{\xi}<\xi_*<\infty$. By the second equation in \eqref{DS} 
we obtain $Q(\bar{\xi})\geq 0$. Arguing as before, we also deduce that $(QZ)(\bar{\xi}) >0$, 
thus we get $Q(\bar{\xi}) >0$ and $Z(\bar{\xi}) >0$. This implies $X(\bar{\xi};\kappa) \in\Sigma_{+}$ 
and {\it (ii)} follows.
\end{proof}
Solutions exhibiting touch-down in finite time enter the invariant region $\Sigma_-$:
\begin{lemma}\label{lem-globalsol}
Let $\kappa\in \R$ and let $X(\cdot;\kappa) = (\Phi,W,Q,Z)$ be the corresponding solution 
for the dynamical system \eqref{DS}. If the maximal time of existence $\xi_M$ is finite, 
then, there exists $\bar \xi < \xi_M$ with $X(\bar \xi;\kappa) \in \Sigma_-$.
\end{lemma}
\begin{proof}
If $\kap < 0$ or $\kap \geq 0$ and $W(\xi) < 0$ for some $\xi \in (0,\xi_M)$, 
then by Lemma~\ref{lem-kapextreme} and Lemma~\ref{lem-bounds}, the solution enters 
$\Sigma_-$ and there is nothing to prove. We hence assume that $\kappa \geq 0$ and 
$W(\xi) \geq 0$ for all $\xi\in [ 0,\xi_{M})$. Then the first equation in \eqref{DS} 
yields $\Phi(\xi) \geq c_{0}e^{-\frac{4\xi}{m}}$ as long as the solution is defined. 
By standard ODE theory, this implies that the solutions of \eqref{DS} are defined globally,
 which gives a contradiction.
\end{proof}

\subsection{Behavior for large $\xi$}

In this section, we consider the asymptotic behavior of solutions which enter 
neither of the two invariant regions $\Sigma_\pm$ (i.e. $\kap \in I_0$). We 
first show that these solutions are globally defined and satisfy certain bounds 
in terms of $\Phi$, $W$:
\begin{lemma} \label{lem-lim1} %
Let $\kappa\in I_{0}$ and let $X(\cdot;\kappa) = (\Phi,W,Q,Z)$ be the corresponding 
solution of the dynamical system \eqref{DS}. Then $X(\cdot, \kap)$ is globally defined with
  \begin{align}\label{limit1}
    \lim_{\xi \to \infty}\Phi(\xi) = K %
    &&\text{where} \quad
       \begin{TC}
         K = 0 \qquad &\text{if $m \in (0,4)$,} \\
         K \geq 0 &\text{if $m=4$}. 
       \end{TC}
  \end{align}
  Furthermore, $W$ satisfies
\begin{align} \label{limit2} %
    0 \leq W(\xi) \leq 2 && \text{for all $\xi \in \R$.}
\end{align}
\end{lemma}
\begin{proof} %
By Lemma~\ref{lem-globalsol}, the solution $X(\cdot;\kap)$ is defined for all $\xi \in \R$. 
By Lemma~\ref{lem-bounds}, \eqref{limit2} holds. It hence remains to show \eqref{limit1}. 
By Lemma~\ref{lem-E2}, $E_2$ is increasing. Hence, one of the following two cases holds: 
Either we have $E_{2}(y) \leq 0$ for all $y > 0$ or there exists $y^*$ such that $E_{2}(y) >0$ 
for all $y > y^*$. We consider these two cases separately:

\medskip

{\it Case $1$:} In this case, we assume that $E_{2}(y) \leq 0$ for all $y >0$. 
Integrating \eqref{D2y} in the interval $\left( 1,\infty\right)$ we then obtain
  \begin{align}\label{diss-fin} %
    \int_{1}^{\infty} \Big( \frac{\alp}{2}\Big( f_{s}-\frac{f}{s} \Big) ^{2} + f^{m}f_{sss}^2  \Big)\ ds \
    =-E_2(1)<\infty,
  \end{align}
since $E_2(y)$ cannot blow up at a finite $y$ (it is increasing and non-positive). 
In particular, we get
  \begin{align} \label{also-fin} %
    \int_{1}^{\infty} \Big|f_s-\frac fs \Big| \, \frac{ds}s %
    \leq\Big( \int_{1}^{\infty}\Big|f_s-\frac fs\Big|^{2}\,ds\Big)^{\frac{1}{2}}
\Big( \int_{1}^{\infty}\frac{ds}{s^{2}}\Big)^{\frac{1}{2}} < \infty.
  \end{align}
Using the observation that $(\frac fy)_y = \frac 1 y(f_y- \frac fy)$, we compute
  \begin{align}\label{calc-wlim} %
    \frac{f(y)}y = f(1) + \int_{1}^{y} \Big(f_s-\frac fs\Big) \frac{ds}{s}.
  \end{align}
In view of \eqref{also-fin}, the right hand side of \eqref{calc-wlim} is absolutely 
convergent as $y \to \infty$ and the limit is finite. Therefore, also the limit 
$K := \lim_{y\to\infty}\frac{f( y) }{y} \geq 0$ exists and is finite. 
Since $\Phi=f/y^{{\frac 4m}}$, in particular we get $\Phi(\xi) \to 0$ as $\xi \to \infty$ 
if $m \in (0,4)$.

\medskip

{\it Case $2$: } In this case, we assume that there exists $y^* > 0$ such that $E_{2}(y) >0$
 for all $y>y^*$. In view of \eqref{E2}, we then have $(QZ)(\xi)>0$ for all $\xi > \xi^*$
 where $y^* = e^{\xi^*}$. In the following, we assume $\xi > \xi^*$. By Lemma~\ref{lem-bounds},
 we also have $0 \leq W(\xi) \leq 2$. We claim that
  \begin{align} \label{W-dicho} %
    \frac{dW}{d\xi} < 0 \quad \text{if} \ W(\xi) > 1 \ && \text{and} && %
    \ \frac{dW}{d\xi} > 0 \quad \text{if} \ W(\xi) < 1\,.
  \end{align}
Indeed, if $W(\xi) > 1$ and $W_\xi(\xi) \geq 0$, then by the second equation in 
\eqref{DS} we get $Q(\xi) >0$. Since $(QZ)(\xi) > 0$, this implies $Z(\xi) > 0$ and 
hence $X(\xi;\kap) \in \Sigma_+$ which contradicts the assumption that $\kap \in I_0$.
 Similarly, if $W(\xi) < 1$ and $W_\xi(\xi) \leq 0$, then by the second equation in \eqref{DS}
 we have $Q(\xi) < 0$. Since $(QZ)(\xi) > 0$, this implies $Z(\xi) < 0$ and hence 
$X(\xi;\kap) \in \Sigma_-$ which, again, contradicts our assumptions. 
Therefore \eqref{W-dicho} holds. It follows that either $0 \leq W(\xi) \leq 1$ 
and $W$ is non decreasing for all $\xi\geq \xi^*$, or that $1 \leq W(\xi) \leq 2$ 
and $W$ is non increasing for all $\xi\geq \xi^*$. Therefore, there is $L\in [0, 2]$ 
such that $\lim_{\xi\to \infty}W(\xi) = L$.

We next show that $L=1$. Arguing by contradiction, we assume that $L\neq1$. 
If $L \in (1,2]$, then as before and by \eqref{W-dicho}, we have $W_\xi(\xi) < 0$ 
and $Q(\xi) < 0$ for all $\xi > \xi^*$. Analogously, if $L \in [0,1)$, then 
$W_\xi(\xi) > 0$ and $Q(\xi) > 0$ for all $\xi > \xi^*$. In both cases, by the
 second equation in \eqref{DS}, we get $|W_\xi| \geq W |W-1|$ and
  \begin{align} \label{dasauch} %
    \int_{\xi^*}^\infty W(\xi) |W(\xi)-1| d\xi \leq 
\int_{\xi^*}^\infty \left| \frac{dW(\xi)}{d \xi}\right| d\xi = 
|W(\xi^*) - L| < \infty,
  \end{align}
which implies $L = 1$ (the case $L=0$ contradicts the dynamics of \eqref{DS}).

Since $W(\xi)\to 1$ as $\xi\to\infty$ and by the first equation in \eqref{DS}, 
we immediately get $\Phi(\xi) \to 0$ as $\xi \to \infty$ if $m \in (0,4)$. 
If $m=4$, by integrating the first equation in \eqref{DS}, we get for $\xi_2 \geq \xi_1 \geq \xi^*$,
\begin{align*}
  \hspace{6ex} & \hspace{-6ex} %
 |\ln \Phi(\xi_1) - \ln \Phi(\xi_2)| = \int_{\xi_1}^{\xi_2} (W(\xi) - 1)d\xi \leq 
\frac 2{W(\xi_1)} \int_{\xi_1}^\infty W(\xi) |W(\xi)-1| d\xi %
 \stackrel{\eqref{dasauch}}{\to}0
 \end{align*}
as $\xi_1 \to \infty$. We have used that $W(\xi) \geq 2W(\xi_1)$ if $W(\xi_1) \geq 1$ 
and that $W(\xi)$ increases monotonically if $W(\xi_1) < 1$. This implies 
$\ln \Phi(\xi) \to \ln K < \infty$ for some $K > 0$ and hence $\Phi(\xi) \to K > 0$ as $\xi\to\infty$.
\end{proof}
%%%%%%%%%%%%%%%%%%%%%%%%%%%%%%%%%%%%%%%%%%%%%
The asymptotic behavior for solutions which enter neither of the two 
invariant regions $\Sigma_\pm$ is described in the next proposition:
\begin{proposition} \label{prp-lim2} %
  Suppose that $\kappa\in I_{0}$ and let $X(\cdot;\kappa) = (\Phi,W,Q,Z)$ be the
  corresponding solution for the dynamical system \eqref{DS}. 
\begin{enumerate}
\item If $m \in (0,4)$, then we have
  \begin{align}\label{limit}
    \lim_{\xi \to \infty}\Phi(\xi) = 0,\ \lim_{\xi \to \infty} W(\xi) = 1,\ \lim_{\xi \to \infty} Q(\xi) = 0,  \ 
    \lim_{\xi \to \infty} Z(\xi) =0.
\end{align}
More precisely, for some $a > 0$ and for $K_0=\frac 38(4-m)^{\frac{4}{3}} a^{\frac{m}{3}}$
 we have
\begin{align}\label{asymptotics} %
    \lim_{\xi\to\infty} e^{\frac{4-m}{m}\xi}\Phi(\xi) = a > 0\,,\ \lim_{\xi\to\infty}e^{K_0 e^{\frac{4-m}{3}\xi} }
    \|(W-1, Q, Z)\| <\infty.
\end{align}

\item If $m=4$, then we have
\begin{align}\label{limit4}
    \lim_{\xi \to \infty}\Phi(\xi) = a >0, \, \lim_{\xi \to \infty} W(\xi) = 1, \, \lim_{\xi \to \infty} Q(\xi) = 0, \, 
    \lim_{\xi \to \infty} Z(\xi) =0.
\end{align}
More precisely, there exist constants $C_1, C_2\in \mathbb{C}$, 
such that for all $\varepsilon>0$ there exists $\xi_0$ large enough satisfying
\begin{equation}\label{asymptoticsII} %
  |W(\xi) - 1 - C_1 e^{z_0\xi} -C_2 e^{z_0^*\xi}|<\varepsilon \quad \mbox{for all} \ \xi>\xi_0,
\end{equation} 
and
\begin{equation}\label{asymptotics4}
  \lim_{\xi\to\infty}e^{|\lambda_0| \xi } \|(W-1, Q, Z)\| <\infty,
 \end{equation}
where $\lambda_0:=\Re(z_0)$. Here, $z_0$ is the root of $(1-z^2)(2+z) + \frac{b}{a^4}$
 with the largest negative real part and, for definiteness, with $\Im(z_0)\geq 0$. Moreover, $|\lambda_0|>1$ for all $a>0$.
\end{enumerate}
\end{proposition}
\begin{proof} %
  {\it (i):} Lemma~\ref{lem-lim1} implies that $\lim_{\xi\to\infty}\Phi(\xi)=0$ and 
$0 \leq W(\xi) \leq 2$ for all $\xi \in \R$. With the notation $Y := (W-1,Q,Z)$, 
we write the last three equations of \eqref{DS} as a linear system:
  \begin{align} \label{DS-Y-xi} %
    \frac{dY}{d\xi} = \Big[\frac A{\Phi^{\frac{m}{3}}} + B(W)\Big] Y,
  \end{align}
  where
  \begin{align} \label{def-AB} %
    A= \left(
      \begin{array} [c]{ccc}%
        0 & 1 & 0\\
        0 & 0 & 1\\
        \alp & 0 & 0
      \end{array}
    \right), \ %
    B(W) =\left(
      \begin{array} [c]{ccc}%
        \hspace{-0.5ex}
        -W & 0 & 0\\
        0 & \frac{2}{3}+\frac{m-3}{3}W & 0\\
        0 & 0 &  \frac{1}{3}-\frac{m+3}{3} W \hspace{-0.5ex} 
      \end{array}
    \right).
  \end{align}
We introduce $\tau(\xi) \in \R$ by $\tau(0) = 0$ and $d\tau=\frac{d\xi}{\Phi^{\frac{m}{3}}}$. 
Then $\Phi^{\frac{m}{3}}\frac{dY}{d\xi} = \frac{dY}{d\tau}$ and $\lim_{\xi\to \infty}\tau(\xi)= \infty$,
 since $\lim_{\xi \to \infty}\Phi(\xi) =0 $. In terms of $\tau$, \eqref{DS-Y-xi} reads
  \begin{align}\label{DS-Y-tau}
    \frac{dY}{d\tau} = [A+\Phi^{\frac{m}{3}} B] Y. %
  \end{align}
For $\tau_0 > 0$ and $\tau \geq \tau_{0}$, we define the operator $U(\tau,\tau_0)$ by
  \begin{align*}
     U(\tau_{0};\tau_{0})=I \quad \text{ and } \quad 
    \frac{d}{d\tau}U(\tau;\tau_{0}) &= [A+\Phi^{\frac{m}{3}} B(\tau)] U(\tau;\tau_{0}).
  \end{align*}
We note that the matrix $B$ is uniformly bounded in $\tau$, since 
$0 \leq W(\tau) \leq 2$. Here and in the following, by an abuse of
 notation, the functions $(\Phi,W,Q,Z)$ are written as functions both of $\xi$
 and of $\tau$, with the obvious meaning. By classical ODE theory,
 $U(\tau;\tau_0)$ is uniquely defined.

Since $\Phi(\tau) \to 0$ as $\tau \to \infty$, we also have 
$\|\Phi^{\frac{m}{3}}(\tau) B(\tau)\| \to 0$. By the continuous dependence of $U(\tau;\tau_{0})$ 
on the initial data at $\tau = \tau_0$, we get
\begin{align}\label{lin:sys:appr}
    \lim_{\tau_{0}\to\infty}\sup_{\| \tau-\tau_{0}\| \leq M}\| U\left( \tau;\tau_{0}\right) -e^{(\tau-\tau_{0}) A}\| =0.
  \end{align}
The distinct eigenvalues $\lam_i \in \C$ of $A$ are given by $\lam_i^3 = \alp$, i.e.
  \begin{align*}
    \lambda_1= \alp^{\frac 13}, && %
    \lambda_2= \big(-\tfrac 12 + \tfrac {\sqrt{3}}2 i \big) \alp^{\frac 13},&& %
    \lambda_3=  \big(-\tfrac 12 - \tfrac {\sqrt{3}}2 i \big) \alp^{\frac 13}.
  \end{align*}
Then, by \eqref{lin:sys:appr} and since $\Phi^{\frac m3} B$ is a uniformly bounded 
given matrix, any solution of \eqref{DS-Y-tau} can be expressed as a linear 
combination of three solutions $Y^k(\tau)$, with
\begin{align*}
 c e^{ \Re \lambda_k \tau - d \int_{\tau_0}^{\tau} \|\Phi(s) ^{\frac m3} B(s)\|ds} \leq \|Y^k_i(\tau;\tau_0)\|   \leq
 C e^{ \Re \lambda_k \tau + d \int_{\tau_0}^{\tau} \|\Phi(s)^{\frac m3}B(s)\|ds},
\end{align*}
here $k,i \in \{ 1,2,3 \}$, for all $\tau \geq \tau_0$ for a sufficiently large $\tau_0$,
and for positive constants $c,C,d > 0$ (see e.g. \cite{Bellman1969}). Using that $\Re (\lam_1) > 0$, 
we conclude that $Y$ is a linear combination of only $Y^2$ and $Y^3$. 
In particular,
\begin{align}\label{est:Y}
 c e^{ - \frac 12 \alp^{1/3}  \tau - d \int_{\tau_0}^{\tau} \Phi^{\frac m3}(s) ds} \leq \|Y(\tau;\tau_0)\|   \leq C
 e^{ - \frac 12 \alp^{1/3}  \tau + d \int_{\tau_0}^{\tau} \Phi^{\frac m3}(s) ds}.
\end{align}
for positive constants $c,C,d > 0$. Since $\lim_{\tau\to\infty}\Phi(\tau)=0$
 for $0 < m < 4$, we obtain $\|Y(\tau;\tau_0)\|\to 0$ as $\tau\to \infty$. 
In particular, assertion \eqref{limit} holds.

\medskip

We turn to the proof of \eqref{asymptotics}. As before, the asymptotic behavior 
of $\Phi$ in the variable $\xi$ is easily obtained from the first equation 
\eqref{DS}, since we know that $W\to 1$ as $\xi\to\infty$. This implies that 
$\lim_{\xi\to\infty} e^{\frac{m-4}{m}\xi} \Phi(\xi)= a$, for some constant $a > 0$ and obtain
 the first statement in \eqref{asymptotics}.

To obtain the asymptotic behavior of $(W,Q,Z)$, we use the asymptotic behavior 
of $\Phi$ together with \eqref{est:Y}.  For $\tau\geq 2\tau_0$, the asymptotic
 behavior of $\Phi$ in $\xi$ and $\tau$ imply 
\begin{align}\label{tau:asymp}
  \frac{2 m}{(4-m)a^{\frac{m}{3}}} e^{\frac{1}{3\alp}(\xi-\xi_0)} \leq \tau-\tau_0 \leq
  \frac{4 m}{(4-m)a^{\frac{m}{3}}} e^{\frac{1}{3\alp}(\xi-\xi_0)},
\end{align}
for $\tau_0 = \tau(\xi_0)$ sufficiently large. Moreover, we can estimate the integral 
in the exponents of \eqref{est:Y} using \eqref{tau:asymp} and the asymptotics 
of $\Phi$ to conclude that
\begin{align}\label{int:phi:tau} %
    C_2 \log(\tau-\tau_0) \leq \int_{\tau_0}^\tau\Phi^{\frac{m}{3}}(s)ds \leq C_1 \log(\tau-\tau_0)
\end{align}
  for $C_1,C_2 > 0$. Using \eqref{tau:asymp} and \eqref{int:phi:tau}
  in \eqref{est:Y}, we obtain that there exists $\xi_0$ large enough such
  that for all $\xi\geq \xi_0$ there exist positive constants with
\begin{align*}
    \hspace{6ex} & \hspace{-6ex} %
                   C_2 e^{- 2\Lambda(\xi)}e^{d_2(\xi-\xi_0)} \leq \| Y(\xi;\xi_0)\| %
    \leq C_1 e^{- \frac 12 \Lambda(\xi)} e^{ d_1(\xi-\xi_0)},
\end{align*}
where $\Lambda(\xi) = \frac 32 \alp^{\frac{4}{3}} a^{-\frac{m}{3}} e^{\frac{1}{3\alp}(\xi-\xi_0)}$. 
This implies the last statement in \eqref{asymptotics}.

  \medskip

{\it (ii):} By Lemma~\ref{lem-lim1}, we have $\Phi(\xi) \to a \geq 0$ as $\xi \to \infty$
 (we rename $K$ to $a$ for consistency with \eqref{LBU}). Let us show first that $a>0$. 
Arguing by contradiction, we assume that $a=0$, then we can apply the same argument as
 in the proof of {\it (i)} with $\alp=b$, and in particular the estimate \eqref{est:Y} also holds. 
This implies
    \begin{align}
      \int_0^\infty |1-W(\xi)| d\xi < \infty.
    \end{align}
    In view of the first equation in \eqref{DS}, this yields $W(\xi)\to 1$ as
    $\xi\to \infty$.  On the other hand, integrating the first equation in
    \eqref{DS}, we get
    \begin{align}
      \Phi(\xi) = \Phi(0) \exp \Big( \int_0^\xi (1- W(\eta)) d\eta \Big) %
      > C > 0,
    \end{align}
    uniformly for $\xi \in (0,\infty)$, but this means that $a > 0$, which is
 a contradiction. Then $a>0$, and the first equation in \eqref{DS}
yields $W(\xi) \to 1$ as $\xi \to \infty$.

\medskip

As in \eqref{DS-Y-xi}, $Y := (W-1,Q,Z)$ satisfies
\begin{equation}\label{DS-Y-xi4}
\frac{dY}{d\xi} =\frac{1}{\Phi^{\frac{4}{3}}} AY + B(W)Y 
=: \frac{1}{a^{\frac{4}{3}}} AY + B(1)Y+ C(\xi)Y,
\end{equation}
where the matrices $A$ and $B$ are given by \eqref{def-AB} (with $\alp=b$). Since $\Phi^{\frac{4}{3}}(\xi) \to a^{\frac{4}{3}}$ and $W(\xi) \to 1$ as $\xi\to\infty$, the matrix $C(\xi)$ satisfies $\|C(\xi)\| \to 0$. 
For $\xi_0 \in \R$ and $\xi \geq \xi_0$, we define the operator $U(\xi,\xi_0)$ by
  \begin{align}\label{U:flow:4}
      U(\xi_{0};\xi_{0})=I \quad \text{ and }\quad %
    \frac{d}{d\xi}U(\xi;\xi_{0}) &= 
\left[\frac{1}{a^{\frac{4}{3}}} A +  B(1)+ C(\xi) \right] U(\xi;\xi_{0}).
  \end{align}
  Since $\|C(\xi)\|\to 0$ as $\xi \to \infty$, by classical ODE theory we have
\begin{align*}
  \lim_{\xi_{0}\to\infty}\sup_{\| \xi-\xi_{0}\| \leq M}\| U( \xi;\xi_{0}) -e^{(\xi-\xi_0) (A/a^{\frac{4}{3}}+B(1))}\|    =0
\end{align*}
for any $M>0$ fixed. In particular, $W(\xi)\to 1$ as $\xi\to\infty$ implies also that $Q(\xi),Z(\xi)\to0$, since this is an
isolated critical point of \eqref{U:flow:4} with $C\equiv 0$. The exponential decay to this point is hence given by the roots with negative real part of the characteristic polynomial of $ A/a^{\frac{4}{3}}+B(1)$, which is
\begin{align*}
  P_{a}(z)= P_0(z) +\frac{b}{a^4}, \qquad \text{where $P_0(z) := (1-z^2)(2+z)$.}
\end{align*}
Clearly, $P_{a}$ always has a positive real root since $b/a^4>0$ and $z=1$ is
the only positive root of $P_0$. Furthermore, since $P_0$ has two real negative
roots and $P_a$ cannot have purely imaginary roots, it follows that $P_a$ has
two roots with negative imaginary part for all $a > 0$ (by continuity in
$a$). Now $\lam_0 = \Re(z_0) < 0$, the real part of the roots $z_0 \in \C$ with
the largest negative real part, controls the exponential decay in
the variable $\xi$. Finally, we note that $|\lambda_0|>1$. If the roots are
real, this is easy to check. If two of the roots are complex conjugates, this follows
from the fact that the sum of the three roots is equal to $-2$ and one of them
is real and positive.
\end{proof}

\medskip

We can now prove the main theorem:
\begin{proof}[Proof of Theorem~\ref{thm-lifting}] 

{\it (i):} From Lemma~\ref{lem-kapextreme} and Lemma~\ref{lem-shooting} we have
  $I_0 \subset (0,\sqrt{12\alp})$ and $I_0\neq \emptyset$. Then Proposition~\ref{prp-lim2} 
gives the behavior of solutions for $\kappa\in I_0$ in terms of the variables 
\eqref{def-PWZQ}. Changing back to the original variables we conclude that 
\eqref{LBU} holds since $e^{\frac{m-4}{m}\xi} \Phi(\xi) = \frac {f(y)}y$ for $y = e^\xi$.

  \medskip
  
{\it (ii):} Let $z_0$ be defined as in Proposition~\ref{prp-lim2}(ii) 
and let $z_0 = \lam_0 + i \omega_0$ for $\omega_0 \geq 0$. 
In view of \eqref{asymptoticsII} we have
\begin{align}
  \Big|(\ln \Phi)_\xi - g(\xi) e^{-|\lam_0| \xi}\Big| = o(e^{-|\lam_0| \xi})
\end{align}
for $g(\xi) = \Re(K e^{-\omega_0i \xi}) = A \cos (\omega_0 \xi) + B\sin(\omega_0 \xi)$
 for some $A,B \in \R$. Integration yields
\begin{align}
    \ln \Phi - \ln a + (A' \cos (\omega_0 \xi) + B' \sin (\omega_0 \xi))  e^{-|\lam_0| \xi} =  o( e^{-|\lam_0| \xi})%
\end{align}
for some $A',B' \in \R$. Since $\Phi = \frac fy$, we get
\begin{align}
 f(y) =  a y e^{\tilde g(y) y^{-|\lam_0|}} + o( y e^{C y^{-|\lam_0|}})  = 
a y + a \tilde g(y) y^{-|\lam_0|+1} + o( y^{-|\lam_0|+1})
\end{align}
for $\tilde g(y) = (A' \cos (\omega_0 \ln y)  + B' \sin (\omega_0\ln y))$.
\end{proof}
As a consequence of the asymptotic behavior of the solutions encoded in 
Proposition~\ref{prp-lim2} we obtain also that the energy of solutions of 
the form \eqref{ss-form} decreases in the following sense:
\begin{corollary}\label{corol-diss}
  \begin{enumerate}
  \item Let $m \in (0,4)$ and $h(x,t)$ be a solution of \eqref{tfe} 
of the form \eqref{ss-form}, where $f$ solves \eqref{ode}--\eqref{LBU}. 
Then, there exist constants $d_0, D_0 > 0$ such that
  \begin{equation}\label{ss-diss}
    -t^\alp d_0 < \frac 12 \int_{-\infty}^\infty (h_x^2 - h_{{\rm in},x}^2)dx\leq 
    -t^\alp D_0.
  \end{equation}
\item Let $m = 4$ and let $h(x,t)$ be a solution of \eqref{tfe} of the form 
\eqref{ss-form4} where $f$ solves \eqref{ode}--\eqref{LBU}. 
Then, there exist constants $d_0, D_0 > 0$ such that
\begin{equation}\label{ss-diss4}
 -e^{b t} d_0 < \frac{1}{2}\int_{-\infty}^\infty (h_x^2 - h_{{\rm in},x}^2)dx\leq -e^{b t} D_0.
  \end{equation}
  \end{enumerate}
\end{corollary}
\begin{proof} 
We only give the proof of (i) since the proof of (ii) proceeds analogously. 
We hence assume $m \in (0,4)$. The dissipation formula reads:
  \begin{equation}\label{diss:form}
    \frac{1}{2}\int_{-\infty}^\infty (h_x^2 - h_{{\rm in},x}^2) \ dx +
 \int_0^t\int_{-\infty}^{\infty} h^m h_{xxx}^2 \ dx  dt = 0.
  \end{equation}
Changing to the self-similar variables in \eqref{diss:form} we obtain
\begin{align*}
 \frac{1}{2}\int_{-\infty}^\infty (h_{{\rm in},x}^2 - h_x^2) \ dx = 
    \frac{2t^{\alp}}{\alp} \int_{0}^{\infty  } f^m f_{yyy}^2 \ dy.
\end{align*}
We now observe that there exist $d_0>0$ and $D_0>0$ such that
\begin{align*}
    d_0 <\int_{0}^{\infty} f^m f_{yyy}^2 \ dy\leq D_0.
\end{align*}
In view of the definitions of $\Phi$, $Z$ in \eqref{def-PWZQ} and by
Proposition~\ref{prp-lim2}(i), we have $\lim_{y \to \infty} \frac{f}{y} = 
\lim_{\xi \to \infty} e^{\frac{4-m}{m}\xi}\Phi(\xi) = a$ and thus
  \begin{align*}
    &\lim_{y\to \infty} e^{K_0y^{\frac{1}{3\alp}}}y^{\frac{2(m-1)}{3}}|f_{yyy}| = %
    \lim_{y \to \infty} (e^{K_0y^{\frac{1}{3\alp}}} y^{\frac{1}{3}}f^{\frac{2m-3}{3}}|f_{yyy}|) 
(y^{\frac{2m-3}{3}}f^{-\frac{2m-3}{3}}) < \infty.
  \end{align*}
Hence, there are $C$, $c$, $c'>0$ and $y_0$ large enough with
\begin{align*}
        c \int_{y_0}^{\infty} e^{-2K_0y^{\frac{1}{3\alp}}}y^{-\frac{1}{3\alp} }\ dy \leq
 \int_{0}^{\infty} f^m(f_{yyy})^2\ dy \leq 
c'+ C \int_{y_0}^{\infty} e^{-2K_0y^{\frac{1}{3\alp}}}y^{-\frac{1}{3\alp} }\ dy.
 \end{align*}
This implies the existence of $d_0>0$ and $D_0>0$ such that
 \eqref{ss-diss} holds.
 \end{proof}

\medskip

\paragraph{Acknowledgments:} The authors acknowledge the financial support of the German Science Foundation 
(DFG) through the project CRC 1060 {\it The mathematics of emergent effects}, of the Spanish Government 
through the MINECO project MTM2014-53145-P, and of the Basque Government through the Research Group grant
 IT641-13.

%%%%%%%%%%%%%%%%%%%%%%%%%%%%%%%%%%%%%%%%%%%%%%%%%%%%%%%%%%%%%%%%%%%%%%

 \appendix
 \section{Merging of two droplets} \label{sec-appendix}

In this section we briefly outline the formal asymptotics that corresponds to the merging
 of two droplets in the case $0 < m < 4$. Namely, we consider \eqref{tfe} 
with initial data %
\begin{align}\label{initial-drops}
  h_{\rm in}(x) = (a |x| - b x^{2})_+
\end{align}
for $a,b > 0$. Notice that $h_{\rm in}$ is a stationary solution of \eqref{tfe} 
representing two identical droplets with a contact touch-down point at
$x=0$. Also the cone conditions
\eqref{separation-point}-\eqref{separation-point-angle} are clearly satisfied at
the contact point $x = 0$. We assume that lifting takes place in the
self-similar way \eqref{ss-form} at $x=0$. To leading order as $t\to 0^+$, the
solution is then given by the self-similar solution of
\eqref{ode}--\eqref{LBU}. The behavior \eqref{LBU}, with exponential
corrections, gives the matching into $h_{\rm in}(x)$ for small
times. In order to derive the next order correction, we write
 \begin{align*}
   h(x,t) = t^\alp f \left( \frac{x}{t^\alp},\log t \right).
 \end{align*}
 Inserting this ansatz into \eqref{tfe} and with the notation 
$y = \frac {|x|}{t^\alp}$, $\tau = \log t$, we get
 \begin{align*}
    f_{\tau}+\alpha (f - yf_{y}) + (f^{m}f_{yyy})_{y} = 0.
 \end{align*}
 To be consistent with \eqref{initial-drops}, we impose the
 matching condition%
  \begin{align} \label{match-1} %
    t^\alp  f( y,\tau) \sim (a |x| -bx^{2})_+ = 
t^\alp (a |y| - b    e^{\alp\tau}y^{2})_+
 \end{align}
 for $|y|\gg 1$, $\tau\to -\infty$, $e^{\tau \alp} y = |x| \ll 1$. Let us denote
 by $f_0$ the solution to \eqref{ode}--\eqref{LBU}. In order to determine the next
 order correction $f_1$, we make the ansatz
 \begin{align*}
    f(y,\tau) = f_0(y) + f_1(y,\tau).
 \end{align*}
 In view of \eqref{match-1} and since $\frac{f_0(y)}{y} \to a$ for
 $|y| \to \infty$, this yields the following matching condition for $f_1$:
 \[
f_1(y,\tau) \sim - be^{\alp\tau}y^{2}\quad\text{as}\ |y| \gg1\,,\quad
 \tau\to-\infty.
\] 
This suggests that we look for functions $f_1$ of the form
\begin{equation}\label{f1-P}
 f_1(y,\tau) =e^{\alp\tau}P(y),
 \end{equation}
 where $P$ has the asymptotic behavior
 \begin{equation}\label{P-match}
 P(y) \sim - by^{2}\quad \text{as}\quad |y|\gg1
 \end{equation}
 and where $P$ solves the leading order balance equation
 \begin{equation}\label{P-eq}
 2 \alp P - yP_{y} + (f_{0}^{m}P_{yyy})_{y} + m(f_{0}^{m-1}f_{0,yyy}P)_{y}=0.
 \end{equation}
 Assuming that there is a solution of \eqref{P-match}--\eqref{P-eq} for a given $b$, 
we obtain the next order correction by solving this problem and \eqref{f1-P}.

{ \bibliographystyle{plain}

%\bibliography{selfsim_lift} 

\begin{thebibliography}{10}

\bibitem{BarrettBloweyGarcke-1998}
J.W. Barrett, J.F. Blowey, and H.~Garcke.
\newblock Finite element approximation of a fourth order nonlinear degenerate
  parabolic equation.
\newblock {\em Numer. Math.}, 80(4):525--556, 1998.

\bibitem{Bellman1969}
R.~Bellman.
\newblock {\em Stability theory in differential equations}.
\newblock Dover Publications, Inc., New York, 1969.

\bibitem{BerettaBertschDalpasso-1995}
E.~Beretta, M.~Bertsch, and R.~Dal~Passo.
\newblock Nonnegative solutions of a fourth-order nonlinear degenerate
  parabolic equation.
\newblock {\em Arch. Rational Mech. Anal.}, 129:175--200, 1995.

\bibitem{BernisFriedman-1990}
F.~Bernis and A.~Friedman.
\newblock Higher order nonlinear degenerate parabolic equations.
\newblock {\em J. Differential Equations}, 83:179--206, 1990.

\bibitem{BernisHulshofKing-2000}
F.~Bernis, J.~Hulshof, and J.~R. King.
\newblock Dipoles and similarity solutions of the thin film equation in the
  half-line.
\newblock {\em Nonlinearity}, 13:413--439, 2000.

\bibitem{BernisPeletierWilliams-1992}
F.~Bernis, L.~A. Peletier, and S.~M. Williams.
\newblock Source type solutions of a fourth order nonlinear degenerate
  parabolic equation.
\newblock {\em Nonlinear Anal.}, 18:217--234, 1992.

\bibitem{BertozziPugh-1996}
A.~L. Bertozzi and M.~Pugh.
\newblock The lubrication approximation for thin viscous films: regularity and
  long-time behavior of weak solutions.
\newblock {\em Comm. Pure Appl. Math.}, 49:85--123, 1996.

\bibitem{Bertozzi-1998}
A.L. Bertozzi.
\newblock The mathematics of moving contact lines in thin liquid films.
\newblock {\em Notices Amer. Math. Soc.}, 45(6):689--697, 1998.

\bibitem{BertschDalpassoGarckeGruen-1998}
M.~Bertsch, R.~Dal~Passo, H.~Garcke, and G.~Gr{\"u}n.
\newblock The thin viscous flow equation in higher space dimensions.
\newblock {\em Adv. Differential Equations}, 3:417--440, 1998.

\bibitem{BertschGiacomelliKarali-2005}
M.~Bertsch, L.~Giacomelli, and G.~Karali.
\newblock Thin-film equations with ``partial wetting'' energy: existence of
  weak solutions.
\newblock {\em Phys. D}, 209:17--27, 2005.

\bibitem{ConstantinDupontEtal-1993}
P.~Constantin, T.~F. Dupont, R.~E. Goldstein, L.~P. Kadanoff, M.~J. Shelley,
  and S.~Zhou.
\newblock Droplet breakup in a model of the {H}ele-{S}haw cell.
\newblock {\em Phys. Rev. E}, 47:4169--4181, Jun 1993.

\bibitem{DalpassoGarckeGruen-1998}
R.~Dal~Passo, H.~Garcke, and G.~Gr{\"u}n.
\newblock On a fourth-order degenerate parabolic equation: global entropy
  estimates, existence, and qualitative behavior of solutions.
\newblock {\em SIAM J. Math. Anal.}, 29:321--342 (electronic), 1998.

\bibitem{DuistermaatKolk2010}
J.~J. Duistermaat and J.~A.~C. Kolk.
\newblock {\em Distributions}.
\newblock Cornerstones. Birkh\"auser Boston, Inc., Boston, MA, 2010.
\newblock Theory and applications, Translated from the Dutch by J. P. van Braam
  Houckgeest.

\bibitem{DussanDavis-1974}
E.~B. Dussan and S.~H. Davis.
\newblock On the motion of a fluid-fluid interface along a solid surface.
\newblock {\em J. Fluid Mech.}, 65:71--95, 1974.

\bibitem{BernisFerreira-1997}
R.~Ferreira and F.~Bernis.
\newblock Source-type solutions to thin-film equations in higher dimensions.
\newblock {\em European J. Appl. Math.}, 8:507--524, 1997.

\bibitem{GiacomelliGnannKnuepferOtto-2014}
L.~Giacomelli, M.~V. Gnann, H.~Kn{\"u}pfer, and F.~Otto.
\newblock Well-posedness for the {N}avier-slip thin-film equation in the case
  of complete wetting.
\newblock {\em J. Differential Equations}, 257:15--81, 2014.

\bibitem{GiacomelliGnannOtto-2013}
L.~Giacomelli, M.~V. Gnann, and F.~Otto.
\newblock Regularity of source-type solutions to the thin-film equation with
  zero contact angle and mobility exponent between {$3/2$} and 3.
\newblock {\em European J. Appl. Math.}, 24:735--760, 2013.

\bibitem{GiacomelliKnuepfer-2010}
L.~Giacomelli and H.~Kn{\"u}pfer.
\newblock A free boundary problem of fourth order: classical solutions in
  weighted {H}\"older spaces.
\newblock {\em Comm. Partial Differential Equations}, 35:2059--2091, 2010.

\bibitem{GiacomelliKnuepferOtto-2008}
L.~Giacomelli, H.~Kn{\"u}pfer, and F.~Otto.
\newblock Smooth zero-contact-angle solutions to a thin-film equation around
  the steady state.
\newblock {\em J. Differential Equations}, 245:1454--1506, 2008.

\bibitem{GiacomelliShishkov-2005}
L.~Giacomelli and A.~Shishkov.
\newblock Propagation of support in one-dimensional convected thin-film flow.
\newblock {\em Indiana Univ. Math. J.}, 54(4):1181--1215, 2005.

\bibitem{Gnann-2016}
M.~V. Gnann.
\newblock On the regularity for the {N}avier-slip thin-film equation in the
  perfect wetting regime.
\newblock {\em Arch. Ration. Mech. Anal.}, 222(3):1285--1337, 2016.

\bibitem{Greenspan-1978}
H.~P. Greenspan.
\newblock Motion of a small viscous droplet that wets a surface.
\newblock {\em J. Fluid Mech.}, 84:125--143, 1978.

\bibitem{Gruen-2004-1}
G.~Gr{\"u}n.
\newblock Droplet spreading under weak slippage---existence for the {C}auchy
  problem.
\newblock {\em Comm. Partial Diff. Eq.}, 29(11-12):1697--1744, 2004.

\bibitem{HuhScriven-1971}
C.~Huh and L.E. Scriven.
\newblock Hydrodynamic model of steady movement of a solid/liquid/fluid contact
  line.
\newblock {\em J. Coll. Int. Sc.}, 35:85--101, 1971.

\bibitem{John-2015}
D.~John.
\newblock On uniqueness of weak solutions for the thin-film equation.
\newblock {\em J. Differential Equations}, 259(8):4122--4171, 2015.

\bibitem{Knuepfer-2011}
H.~Kn{\"u}pfer.
\newblock Navier slip thin--film equation for partial wetting.
\newblock {\em Comm. Pure Appl. Math.}, 64:1263--1296, 2011.

\bibitem{KnuepferMasmoudi-2013}
H.~Kn{\"u}pfer and N.~Masmoudi.
\newblock Well-posedness and uniform bounds for a nonlocal third order
  evolution operator on an infinite wedge.
\newblock {\em Comm. Math. Phys.}, 320(2):395--424, 2013.

\bibitem{KnuepferMasmoudi-2015}
H.~Kn\"upfer and N.~Masmoudi.
\newblock Darcy's flow with prescribed contact angle: well-posedness and
  lubrication approximation.
\newblock {\em Arch. Ration. Mech. Anal.}, 218(2):589--646, 2015.

\bibitem{OronDavisBankoff-1997}
A.~Oron, S.~Davis, and S.~Bankoff.
\newblock Long-scale evolution of thin liquid films.
\newblock {\em Rev. Mod. Phys.}, 69(3):931--980, 1997.

\bibitem{Otto-1998}
F.~Otto.
\newblock Lubrication approximation with prescribed nonzero contact angle.
\newblock {\em Comm. Partial Differential Equations}, 23:2077--2164, 1998.

\bibitem{Reynolds-1886}
O.~Reynolds.
\newblock On the theory of lubrication and its application to {M}r. {B}eauchamp
  {T}ower's experiments, including an experimental determination of the
  viscosity of olive oil.
\newblock {\em Proc. R. Soc. London}, 40:191--203, 1886.

\end{thebibliography}
}

\end{document}